\documentclass[a4paper,11pt]{amsart}

\usepackage{amsmath, amsthm, amsfonts, amssymb}
\usepackage[bookmarks=false]{hyperref}
\usepackage{enumerate}

\usepackage{color} 

\numberwithin{equation}{section}

\newtheorem*{mainthm}{Main theorem}

\newtheorem{thm}{Theorem}[]
\newtheorem{lem}[thm]{Lemma}

\theoremstyle{definition}
\newtheorem{rem}[thm]{Remark}
\newtheorem*{rem*}{Remark}

\newtheorem{df}[thm]{Definition}

\newcommand{\R}{\mathbf{R}}
\newcommand{\C}{\mathbf{C}}

\newcommand{\N}{\mathbf{N}}

\newcommand{\cU}{\mathcal{U}}

\newcommand{\rL}{\mathord{\text{\rm L}}}
\newcommand{\HNN}{\mathord{\text{\rm HNN}}}
\newcommand{\rF}{\mathord{\text{\rm F}}}

\newcommand{\rC}{\mathord{\text{\rm C}}}

\newcommand{\rT}{\mathord{\text{\rm T}}}
\newcommand{\rS}{\mathord{\text{\rm S}}}
\newcommand{\rE}{\mathord{\text{\rm E}}}

\newcommand{\spn}{\mathord{\text{\rm span}}}

\newcommand{\diag}{\operatorname{diag}}

\newcommand{\cR}{\mathcal{R}}

\begin{document}

\title[Amenable absorption in amalgamated free products]{Amenable absorption in  amalgamated free product \\ von Neumann algebras}

\begin{abstract}
We investigate the position of amenable subalgebras in arbitrary amalgamated free product von Neumann algebras $M = M_1 \ast_B M_2$. Our main result states that under natural analytic assumptions, any amenable subalgebra of $M$ that has a large intersection with $M_1$ is actually contained in $M_1$. The proof does not rely on Popa's asymptotic orthogonality property but on the study of non normal conditional expectations. 
\end{abstract}

\author{R\'emi Boutonnet}
\address{Institut de Math\'ematiques de Bordeaux \\ CNRS \\ Universit\'e Bordeaux I \\ 33405 Talence \\ FRANCE}
\email{remi.boutonnet@math.u-bordeaux1.fr}

\author{Cyril Houdayer}
\address{Laboratoire de Math\'ematiques d'Orsay\\ Universit\'e Paris-Sud\\ CNRS\\ Universit\'e Paris-Saclay\\ 91405 Orsay\\ FRANCE}
\email{cyril.houdayer@math.u-psud.fr}
\thanks{Research supported by ERC Starting Grant GAN 637601}


\maketitle

\section*{Introduction}

In his breakthrough article \cite{Po83}, Popa introduced a powerful method based on asymptotic orthogonality in the ultraproduct framework to prove maximal amenability results in tracial von Neumann algebras. Notably, Popa showed in \cite{Po83} that the generator masa in any free group factor is maximal amenable, thus solving an open problem raised by Kadison. 

The question of proving maximal amenability results in von Neumann algebras has attracted a lot of interest over the last few years. Let us single out two recent results related to the present work. Houdayer--Ueda \cite{HU15} completely settled the question of maximal amenability of the inclusion $M_1 \subset M$ in arbitrary free product von Neumann algebras $(M, \varphi) = (M_1, \varphi_1) \ast (M_2, \varphi_2)$. Using a method based on the study of central states, Boutonnet--Carderi \cite{BC14} proved maximal amenability results in (tracial) von Neumann algebras arising from amalgamated free product groups, among other things. We refer the reader to \cite{BC14, HU15} and the references therein for other recent maximal amenability results.

In this paper, we use yet another method, inspired by \cite{BC14}, based on the study of non normal conditional expectations to prove maximal amenability results in arbitrary amalgamated free product von Neumann algebras. We say that an inclusion $P \subset N$ of von Neumann algebras is {\em with expectation} if there exists a faithful normal conditional expectation $\rE_P : N \to P$. We refer to Section~\ref{section-intertwining} for Popa's intertwining-by-bimodules in arbitrary von Neumann algebras.

\begin{mainthm}
For each $i \in \{1, 2\}$, let $B \subset M_i$ be any inclusion of arbitrary $\sigma$-finite von Neumann algebras with expectation. Denote by $M = M_1 \ast_B M_2$ the corresponding amalgamated free product von Neumann algebra. Let $Q \subset M$ be any von Neumann subalgebra with expectation satisfying the following two conditions:
\begin{enumerate}
\item [$(\rm i)$] $Q$ is amenable relative to $M_1$ inside $M$ (e.g.\ $Q$ is amenable).
\item [$(\rm ii)$] $Q \cap M_1 \subset M_1$ is with expectation and $Q \cap M_1 \npreceq_{M_1} B$.
\end{enumerate}
Then $Q \subset M_1$.
\end{mainthm}

We point out that the idea of \cite{BC14} was also adapted in \cite{Oz15} to prove the above theorem in the context of tracial free products. Our strategy follows a different path and is valid in the non-tracial setting and allows the presence of arbitrary amalgams.

Our main theorem completely settles the question of maximal amenability of the inclusion $M_1 \subset M$ in arbitrary amalgamated free product von Neumann algebras $M = M_1 \ast_B M_2$. Our result strengthens and recovers \cite[Corollary B]{HU15} (with a new and much simpler proof). It also generalizes a result obtained by Leary for certain tracial amalgamated free products  \cite{Le14}. A corollary to our main theorem is that if $B$ is of type ${\rm I}$ and $M_1$ has no type ${\rm I}$ direct summand or if $B$ is semifinite and $M_1$ is of type ${\rm III}$, then $M_1$ is maximal amenable (with expectation) inside $M$ whenever $M_1$ is amenable.

Let us give a few comments on the proof of our main theorem. If one tries to use Popa's central sequence approach \cite{Po83}, a key fact one has to show is that $Q$-central sequences in $M$ have no mass on the closed subspace $\mathcal K \subset \rL^2(M)$ spanned by all the reduced words in $M$ starting with a letter in $M_2 \ominus B$. In the setting of free products, this fact is proven by making $\mathcal K$ almost orthogonal to $u_n \mathcal K u_n^*$ where $u_n \in \mathcal U(Q \cap M_1)$ is a well chosen sequence of unitaries witnessing that $Q \cap M_1$ is diffuse. In the presence of the nontrivial amalgam $B$, this is no longer possible in general, even with the stronger assumption that $Q \cap M_1 \npreceq_{M_1} B$. So Popa's strategy via central sequences cannot work for arbitrary amalgamated free products.

To overcome this difficulty, we employ the central state formalism from \cite{BC14} which is better suited for analytic arguments. In the tracial setting, our proof boils down to showing that any $Q$-central state on $\mathbf B(\rL^2(M))$ vanishes on the orthogonal projection $P_{\mathcal K} : \rL^2(M) \to \mathcal K$ (where $\mathcal K$ is as above). To do this, we use a key vanishing type result for central states due to Ozawa--Popa \cite[Lemma 3.3]{OP08}, whose proof relies on $\rC^*$-algebraic techniques.

To run the argument for arbitrary amalgamated free products, we work with conditional expectations $\Phi: \mathbf B(\rL^2(M)) \to Q$ rather than $Q$-central states. We prove a characterization of Popa's intertwining-by-bimodules \cite{Po01, Po03} for arbitrary von Neumann algebras in terms of bimodular completely positive maps (see Theorem \ref{thm-intertwining} below, whose proof relies on a combination of results from \cite{Ha77a, Ha77b, HI15, OP07}).

\subsection*{Acknowledgments} The authors are grateful to the Hausdorff Research Institute for Mathematics (HIM) in Bonn for their kind hospitality during the program {\em von Neumann algebras} where this paper was completed. The second named author thanks Amine Marrakchi for pointing out to him the references \cite{Ha77a, Ha77b} regarding Theorem \ref{thm-intertwining}. The authors also thank Pierre Fima, Fran\c{c}ois Le Ma\^itre and Todor Tsankov for interesting remarks on this work and  the anonymous referee for providing useful comments.

\section{Preliminaries}


For any von Neumann algebra $M$, we denote by $(M, \rL^2(M), J, \rL^2(M)_+)$ its standard form, by $\mathcal Z(M)$ its center and by $\mathcal U(M)$ its group of unitaries The standard Hilbert space $\rL^2(M)$ has a natural structure of $M$-$M$-bimodule and we simply write $x\xi y := xJy^*J \xi$ for all $\xi \in \rL^2(M)$ and all $x, y \in M$. For any faithful state $\varphi \in M_\ast$, denote by $\xi_\varphi \in \rL^2(M)_+$ the unique canonical vector implementing $\varphi \in M_\ast$. Write $\|x\|_\varphi = \|x \xi_\varphi\|_{\rL^2(M)}$ for every $x \in M$. For any projection $p \in M$, denote by $z_M(p) \in \mathcal Z(M)$ its central support in $M$.

\subsection*{Amalgamated free product von Neumann algebras}

For each $i \in \{ 1, 2 \}$, let $B \subset M_i$ be any inclusion of $\sigma$-finite von Neumann algebras with faithful normal conditional expectation $\rE_i : M_i \to B$. The {\em amalgamated free product} $(M, \rE) = (M_1, \rE_1) \ast_B (M_2, \rE_2)$ is a pair of von Neumann algebra $M$ generated by $M_1$ and $M_2$ and a faithful normal conditional expectation $\rE : M \to B$ such that $M_1, M_2$ are {\em freely independent} with respect to $\rE$:
$$\rE(x_1 \cdots x_n) = 0 \quad \text{whenever } \; x_j \in M_{i_j}^\circ \; \text{ and } \; i_1 \neq \cdots \neq  i_{n}.$$
Here and in what follows, we denote by $M_i^\circ := \ker(\rE_i)$. We refer to the product $x_1 \cdots x_n$ where $x_j \in M_{i_j}^\circ$ and $i_1 \neq \cdots \neq i_{n}$ as a {\em reduced word} in $M_{i_1}^\circ \cdots M_{i_n}^\circ$ of {\em length} $n \geq 1$. The linear span of $B$ and of all the reduced words in $M_{i_1}^\circ \cdots M_{i_n}^\circ$ where $n \geq 1$ and $i_1 \neq \cdots \neq i_{n}$ forms a unital $\sigma$-strongly dense $\ast$-subalgebra of $M$. We call the resulting $M$ the \emph{amalgamated free product von Neumann algebra} of $(M_1,\rE_1)$ and $(M_2,\rE_2)$. 

Let $\varphi \in B_\ast$ be any faithful state. Then for all $t \in \R$, we have $\sigma_t^{\varphi \circ \rE} = \sigma_t^{\varphi \circ \rE_1} \ast_B \sigma_t^{\varphi \circ \rE_2}$ (see \cite[Theorem 2.6]{Ue98}). By \cite[Theorem IX.4.2]{Ta03}, for every $i \in \{1, 2\}$, there exists a unique $(\varphi\circ \rE)$-preserving conditional expectation $\rE_{M_i} : M \to M_i$. Moreover, we have $\rE_{M_i}(x_1 \cdots x_n) = 0$ for all the reduced words $x_1 \cdots x_n$ that contain at least one letter from $M_j^\circ$ for some $j \neq i$ (see e.g.\ \cite[Lemma 2.1]{Ue10}). We will denote by $M \ominus M_i := \ker (\rE_{M_i})$. For more information on amalgamated free product von Neumann algebras, we refer the reader to \cite{Ue98, VDN92}.

\subsection*{Relative amenability}

Let $M$ be any von Neumann algebra and $P, Q \subset M$ any von Neumann subalgebras with  expectation. Denote by $\langle M, Q\rangle := (JQJ)' \subset \mathbf B(\rL^2(M))$ the {\em Jones basic construction} of the inclusion $Q \subset M$. Following \cite[Theorem 2.1]{OP07}, we say that $P$ {\em is amenable relative to} $Q$ {\em inside} $M$ if there exists a conditional expectation $\Phi : \langle M, Q\rangle \to P$ such that $\Phi |_M$ is faithful and normal.

Observe that if $P \subset M$ is with expectation and amenable (hence {\em injective}), then by Arveson's extension theorem, $P$ is amenable relative to any von Neumann subalgebra with expectation $Q \subset M$ inside $M$.

\section{Intertwining-by-bimodules for arbitrary von Neumann algebras}\label{section-intertwining}

Popa introduced his powerful intertwining-by-bimodules in the case when the ambient von Neumann algebra is tracial \cite{Po01, Po03}. This intertwining-by-bimodules has recently been adapted to the type ${\rm III}$ setting by Houdayer--Isono \cite{HI15}.

We will use the following notation throughout this section. Let $M$ be any $\sigma$-finite von Neumann algebra and $A \subset 1_A M 1_A$ any $B \subset 1_B M 1_B$ any von Neumann subalgebras with expectation. Let $(M, \rL^2(M), J, \rL^2(M)_+)$ be the standard form of $M$. Define $\widetilde B := B \oplus \C 1_B^\perp$ and observe that $\widetilde B \subset M$ is with expectation. Fix a faithful normal conditional expectation $\rE_{\widetilde B} : M \to \widetilde B$. Regard $\rL^2(\widetilde B) \subset \rL^2(M)$ as a closed subspace via the mapping $\rL^2(\widetilde B)_+ \to \rL^2(M)_+ : \xi_\varphi \mapsto \xi_{\varphi \circ \rE_{\widetilde B}}$. 
The {\em Jones projection} $e_{\widetilde B} : \rL^2(M) \to \rL^2(\widetilde B)$ satisfies
\[J 1_B J e_{\widetilde B} = 1_B e_{\widetilde B} = e_{\widetilde B} 1_B = e_{\widetilde B} J 1_B J.\]
We will denote by $\langle M, \widetilde B\rangle := (J\widetilde BJ)' \subset \mathbf B(\rL^2(M))$ the {\em Jones basic construction} and by $\rT : \langle M, \widetilde B\rangle_+ \to \widehat{M}_+$ the canonical faithful normal semifinite operator valued weight which satisfies $\rT(e_{\widetilde B}) = 1$. We refer the reader to \cite{Ha77a, Ha77b} for more information on operator valued weights. 

\begin{df}[{\cite[Definition 4.1]{HI15}}]\label{df-intertwining} 
We say that $A$ {\em embeds with expectation into} $B$ {\em inside} $M$ and write $A \preceq_M B$, if there exist projections $e \in A$ and $f \in B$, a nonzero partial isometry $v \in eMf$ and a unital normal $\ast$-homomorphism $\theta: eAe \to fBf$ such that the inclusion $\theta(eAe) \subset fBf$ is with expectation and $av = v\theta(a)$ for all $a \in eAe$.
\end{df}

We now provide a criterion for $A \preceq_M B$ in terms of (normal) bimodular completely positive maps, that generalizes part of \cite[Theorem 4.3]{HI15}. Note that there is no restriction on the type of any of the algebras involved.

\begin{thm}\label{thm-intertwining}
Keep the same notation as above. The following assertions are equivalent:
\begin{itemize}
\item [$(\rm i)$] $A \preceq_M B$.
\item [$(\rm ii)$] There exists a nonzero element $d \in A' \cap (1_A\langle M, \widetilde B\rangle 1_A)_+$ such that $d \, 1_A J1_B J = d$ and $\rT (d) \in M_+$.
\item [$(\rm iii)$] There exists a normal $A$-$A$-bimodular completely positive map $\Phi : \langle M, \widetilde B\rangle  \to A$ such that $\Phi(1_AJ1_BJ) \neq 0$.
\item [$(\rm iv)$] There exists an $A$-$A$-bimodular completely positive map $\Psi : \langle M, \widetilde B\rangle \to A$ such that $\Psi |_{M}$ is normal and $\Psi |_{1_AM e_{B} M1_A} \neq 0$ where $e_B := e_{\widetilde B} J1_BJ$. 
\end{itemize}
\end{thm}

\begin{proof}
$(\rm i) \Rightarrow (\rm ii)$ Let $e, f, v, \theta$ be as in Definition \ref{df-intertwining} witnessing that $A \preceq_M B$. Define the element $c := v e_{\widetilde B} v^* \neq 0$. Then $c \in (eAe)' \cap (e \langle M, \widetilde B \rangle e)_+$, and $\rT(c) = v \, \rT(e_{\widetilde B}) \, v^* = vv^* \in M_+$. Moreover, since $v =  vf = v1_B$, we have $J1_BJc = v J1_BJ e_{\widetilde B}v^* = v 1_B e_{\widetilde B}v^* = c$.

Next, choose a countable family of partial isometries $(w_n)_{n \in \N}$ in $A$ such that $w_n^*w_n \leq e$ for every $n \in \N$ and $\sum_{n \in \N} w_n w_n^* = z_A(e)$, where $z_A(e)$ denotes the central support of $e$ in $A$.  We may assume without loss of generality that $w_1 = e$. Put $d := \sum_{n \in \N} w_n c w_n ^* = \sum_{n \in \N} w_n v e_{\widetilde B} v^* w_n^*$. A simple calculation shows that $d \in A' \cap (1_A \langle M, \widetilde B\rangle 1_A)_+$, $d \geq c \neq 0$ and $\rT(d) = \sum_{n \in \N} w_n vv^* w_n^* \in M_+$. Moreover, $dJ1_BJ = d$ since the same holds for $c$, and each $w_n$ commutes with $J1_BJ$.

$(\rm ii) \Rightarrow (\rm i)$  We may choose a suitable nonzero spectral projection $p$ of $d$ such that $p \in A' \cap 1_A \langle M, \widetilde B \rangle 1_A$, $\rT(p) \in M_+$ and $p \, 1_A \, J1_B J = p$. By applying \cite[Lemma 2.2]{HI15}, there exists a nonzero projection $q \in \mathcal Z(A)p$ such that the inclusion $Aq \subset q \langle M, \widetilde B\rangle q$ is with expectation (see also \cite[Theorem 6.6 (iv)]{Ha77b}). Put $e_B := e_{\widetilde B} J1_BJ$ and $r := q \vee e_{B} \in  \langle M, \widetilde B\rangle$. Since $q$ and $e_{ B}$ are $\sigma$-finite projections in $\langle M, \widetilde B\rangle$ so is $r$ in $\langle M, \widetilde B\rangle$ and hence $r\langle M, \widetilde B\rangle r$ is a $\sigma$-finite von Neumann algebra. We then obviously have $q \langle M, \widetilde B\rangle q \preceq_{r\langle M, \widetilde B\rangle r} r\langle M, \widetilde B\rangle r$. Since the central support of $e_{B}$ in $\langle M, \widetilde B\rangle$ is equal to $J1_B J$ and since $q \, J1_B J = q$, the central support of $e_B$ in $r\langle M, \widetilde B\rangle r$ is equal to $r$. Since $e_{B} \, \langle M, \widetilde B\rangle \, e_{B}  = B e_{B}$, we have $q \langle M, \widetilde B\rangle q \preceq_{r\langle M, \widetilde B\rangle r} B e_{B}$ by \cite[Remark 4.5]{HI15}. Since the inclusion $Aq \subset q \langle M, \widetilde B\rangle q$ is with expectation, we finally have $Aq \preceq_{r \langle M, \widetilde B\rangle r} B e_{B}$ by \cite[Lemma 4.8]{HI15}.

Then there exist projections $e \in A$ and $f \in B$, a nonzero partial isometry $V \in eq \, \langle M, \widetilde B\rangle f e_B$ and a unital normal $\ast$-homomorphism $\theta : eAe \to fBf$ such that the unital inclusion $\theta(eAe) \subset fBf$ is with expectation and $a V = V \theta(a)$ for all $a \in A$ (observe that the $\ast$-homomorphism $fBf \to fBf e_B : y \mapsto y  e_B$ is injective). Thus, we have $\theta(a) V^* = V^* a$ for all $a \in A$. We now follow the lines of the proof of \cite[Theorem 4.3 $(6) \Rightarrow (1)$]{HI15} and use the same notation. Since $(V^*)^*V^* \leq q$ and $\rT(q) \in M$, we have $V^* \in \mathfrak n_{\rT}$. Since $ e_B V^* = V^*$ and $ e_B \in \mathfrak n_{\rT}$, we also have that $V^* \in \mathfrak m_{\rT}$. We may apply $\rT$ to the equation $\theta(a) V^* = V^* a$ and we obtain that $\theta(a) \rT(V^*) = \rT(V^*) a$ for all $a \in A$. Since $V^* =  e_B V^* =  e_B \rT( e_B V^*) =  e_B \rT(V^*)$ by \cite[Proposition 2.5]{HI15} and since $V^* \neq 0$, we have $\rT(V^*) \neq 0$. Finally \cite[Remark 4.2 (1)]{HI15} shows that $A \preceq_M B$.

$(\rm ii) \Rightarrow (\rm iii)$ The mapping $\Phi : \langle M, \widetilde B\rangle_+ \to A_+ : x \mapsto \rE_A(\rT(d^{1/2} x d^{1/2}))$ extends to a well-defined normal $A$-$A$-bimodular completely positive map $\Phi : \langle M, \widetilde B\rangle \to A$ such that $\Phi(1_A J1_B J) = \rE_A(\rT(d)) \neq 0$ (see \cite[Lemma 4.5]{Ha77a}).

$(\rm iii) \Rightarrow (ii)$ Define the nonzero normal bounded operator valued weight $\rS : 1_A\langle M, \widetilde B\rangle 1_A \to A : T \mapsto \Phi(T \, J1_B J)$ and denote by $p \in A' \cap 1_A\langle M, \widetilde B\rangle 1_A$ the support projection of $S$. Denote by $z \in \mathcal Z(A)$ the unique central projection such that $Az^\perp = \ker(A \to Ap : a \mapsto ap)$. Observe that $zp = p$ and the $\ast$-homomorphism $Az \to Ap : az \mapsto azp$ is injective. The mapping $\rS_p : p \langle M, \widetilde B\rangle p \to Ap : x \mapsto S(x)p$ is a faithful normal bounded operator valued weight. Likewise, the mapping $\rT_p : p \langle M, \widetilde B\rangle p \to Ap : x \mapsto \rE_A(\rT(x))p$ is a faithful normal semifinite operator valued weight. By \cite[Theorem 6.6 (ii)]{Ha77b}, $\rT_p$ is still semifinite on $p  (A' \cap 1_A\langle M, \widetilde B\rangle 1_A) p$ and hence there exists a nonzero element $c \in p(A' \cap 1_A\langle M, \widetilde B\rangle 1_A)p$ such that $\rE_A(\rT(c))p = \rT_p(c) \in (Ap)_+$. This implies that $\rE_A(\rT(c)) = \rE_A(\rT(cz)) = \rE_A(\rT(c))z \in A_+$. Thus, there exists a nonzero element $d \in A' \cap (1_A\langle M, \widetilde B\rangle 1_A)_+$ such that $\rT(d) \in M_+$.

$(\rm iii) \Rightarrow (\rm iv)$ This implication is obvious.

$(\rm iv) \Rightarrow (\rm iii)$ Put $e_B := e_{\widetilde B} J1_BJ = e_{\widetilde B}1_B$. Denote by $\Lambda$ the set of triplets $(\varepsilon, \mathcal F, \mathcal G)$ where $0 < \varepsilon < 1$, $\mathcal F \subset \rC^*(M e_{B} M)$ is a nonempty finite subset such that $\mathcal F = \mathcal F^*$ and $\mathcal G \subset M$ is a nonempty finite subset. Define the order relation $\leq$ on $\Lambda$ by
$$(\varepsilon_1, \mathcal F_1, \mathcal G_1) \leq (\varepsilon_2, \mathcal F_2, \mathcal G_2) \quad \text{if and only if} \quad \varepsilon_2 \leq \varepsilon_1, \, \mathcal F_1 \subset \mathcal F_2, \, \mathcal G_1 \subset \mathcal G_2.$$
Then $(\Lambda, \leq)$ is a directed set. Following the lines of the proof of \cite[Lemma 3.3]{OP08}, since $M$ lies in the multiplier algebra of $\rC^*(M e_{B} M)$ inside $\mathbf B(\rL^2(M))$ and  using \cite[Proposition I.9.16]{Da96}, for every $\lambda = (\varepsilon, \mathcal F, \mathcal G) \in \Lambda$, we may choose $g_\lambda \in \rC^*(M e_{B} M)$ such that $0 \leq g_\lambda \leq 1$, $\|g_\lambda x - x\| < \varepsilon$ for every $x \in \mathcal F$ and $\|g_\lambda y - y g_\lambda\| < \varepsilon$ for every $y \in \mathcal G$. Since $\spn (M e_{B} M)$ is dense in $\rC^*(M e_{B} M)$, for every $\lambda = (\varepsilon, \mathcal F, \mathcal G) \in \Lambda$, we may find an element $h_\lambda \in \spn (M e_{B} M)$ such that $\|h_\lambda - g_\lambda^{1/2} \| < \varepsilon$ and $\Vert h_\lambda \Vert \leq 1$. For every $\lambda = (\varepsilon, \mathcal F, \mathcal G) \in \Lambda$, the element $f_\lambda := h_\lambda^*h_\lambda$ belongs to $\spn(M e_B M)$ and satisfies $0 \leq f_\lambda \leq 1$ and
\[\|f_\lambda - g_\lambda\| \leq \|h_\lambda^*h_\lambda - h^*_\lambda g_\lambda^{1/2} \| + \|h_\lambda^*g_\lambda^{1/2} - g_\lambda\|  \leq 2 \varepsilon.\]
In particular, we have that  $\lim_{\lambda} \|f_\lambda x - x \| = 0$ for every $x \in \rC^*(M e_B M)$ and $\lim_{\lambda} \|f_\lambda y - y f_\lambda\| = 0$ for every $y \in M$.

Define the completely positive map $\Phi_\lambda : \langle M, \widetilde B\rangle \to A : T \mapsto \Psi(f_\lambda T f_\lambda)$ and denote by $\Phi$ a pointwise $\sigma$-weak limit of $(\Phi_\lambda)_{\lambda \in \Lambda}$. Namely, fix a cofinal ultrafilter $\mathcal U$ on the directed set $\Lambda$ and define $\Phi(T) = \sigma\text{-weak} \lim_{\lambda \to \cU} \Phi_\lambda(T)$ for every $T \in \langle M,\widetilde B\rangle$. From the properties of $(f_\lambda)$, we see that $\Phi |_{\rC^*(M e_{B} M)} = \Psi |_{\rC^*(M e_{B} M)}$ and $\Phi$ is an $A$-$A$-bimodular completely positive map. Moreover,
\[ \Phi |_{1_AJ1_BJ \, 1_AMe_BM1_A} =  \Phi |_{1_AMe_BM1_A} = \Psi |_{1_AMe_BM1_A} \neq 0.\]
Using the multiplicative domain of $\Phi$, we have that $\Phi(1_AJ1_BJ) \neq 0$. Our task is now to show that $\Phi$ is normal.

Firstly, we claim that $\Phi_\lambda$ is normal for every $\lambda \in \Lambda$. Indeed, put $f_\lambda = \sum_{i = 1}^k x_i e_{B} y_i$ for some $x_1, \dots, x_n, y_1, \dots, y_n \in M$. For every $T \in (1_A\langle M, \widetilde B\rangle1_A)_+$, since $[\rE_{\widetilde B}(x_i^* T x_j)]_{i, j = 1}^k \in \mathbf M_k(\widetilde B)_+$ commutes with $\diag(e_{B}, \dots, e_{B}) = \diag(e_{\widetilde B} J1_BJ, \dots, e_{\widetilde B}J1_BJ)$, we have
$$f_\lambda T f_\lambda = f_\lambda^* T f_\lambda = \sum_{i, j = 1}^k y_i^*  \rE_{\widetilde B}(x_i^* T x_j)e_{B}  y_j  \leq \sum_{i, j = 1}^k y_i^*  \rE_{\widetilde B}((x_i1_B)^* \, T \, (x_j 1_B)) y_j.$$
Since $\Phi$ is completely positive, we have
$$\Phi_\lambda(T) = \Psi(f_\lambda T f_\lambda) \leq \Psi |_M \left ( \sum_{i, j = 1}^k y_i^* \, \rE_{\widetilde B}((x_i1_B)^* \, T \, (x_j1_B)) \, y_j\right ).$$
Since $\Psi|_M$ is normal, this implies that $\Phi_\lambda$ is normal for every $\lambda \in \Lambda$. 

In order to deduce that $\Phi$ is normal, take $\varphi \in (A_\ast)_+$. Since $\Phi$ and $\Psi$ coincide on $\rC^*(M e_{B} M)$ and since $f_\lambda \in \spn (M e_{B} M)$ and $0 \leq f_\lambda \leq 1$, we have 
\[(\varphi \circ \Phi)(1) \geq \lim_{\lambda \to \mathcal U} (\varphi \circ \Phi)(f_\lambda) = \lim_{\lambda \to \mathcal U} (\varphi \circ \Psi)(f_\lambda) \geq \lim_{\lambda \to \mathcal U} (\varphi \circ \Psi)(f_\lambda^2) = (\varphi \circ \Phi)(1).\]

Cauchy--Schwarz inequality implies that 
\begin{align*}
\lim_{\lambda \to \mathcal U} \|\varphi \circ \Phi - (\varphi \circ \Phi)(f_\lambda \, \cdot \, f_\lambda)\| & \leq \lim_{\lambda \to \mathcal U} 2|(\varphi \circ \Phi)((1 - f_\lambda)^2) |^{1/2} \\
& \leq \lim_{\lambda \to \mathcal U} 2|(\varphi \circ \Phi)(1 - f_\lambda) |^{1/2} = 0.
\end{align*}

For every $\lambda \in \Lambda$ and every $T \in \langle M, \widetilde B\rangle$, we have $f_\lambda Tf_\lambda \in \spn(M e_{B} M)$ and hence $\Phi(f_\lambda\,  \cdot \, f_\lambda ) = \Psi(f_\lambda \, \cdot \, f_\lambda) = \Phi_\lambda$. We get that $\lim_{\lambda \to \mathcal U} \|\varphi \circ \Phi - \varphi \circ \Phi_\lambda\| = 0$. Since $\varphi \circ \Phi_\lambda$ is normal for every $\lambda \in \Lambda$, it follows that $\varphi \circ \Phi$ is normal. Since this holds true for every $\varphi \in (A_\ast)_+$, we obtain that $\Phi$ is normal.
\end{proof}

\section{Proof of the main theorem}

Throughout this section, we keep the same notation as in the statement of the main theorem. Write $(M, \rE) = (M_1, \rE_1) \ast_B (M_2, \rE_2)$ for the amalgamated free product von Neumann algebra. Fix a faithful state $\varphi \in M_\ast$ such that $\varphi = \varphi \circ \rE$. Denote by $\mathcal K$  the closure in $\rL^2(M)$ of the linear span of all the elements of the form $x_1 \cdots x_n \xi_\varphi$ where $n \geq 1$ and $x_1 \cdots x_n \in M$ is a reduced word starting with a letter $x_1 \in M_2^\circ$. Using \cite[Section 2]{Ue98}, $\mathcal K$ is naturally endowed with a structure of $B$-$M_1$-bimodule and as $M_1$-$M_1$-bimodules, we have the following isomorphism
$$\rL^2(M) \ominus \rL^2(M_1) \cong \rL^2(M_1) \otimes_B \mathcal K.$$
We will identify $\rL^2(M) \ominus \rL^2(M_1)$ with $\rL^2(M_1) \otimes_B \mathcal K$ and $\mathcal K$ with $\rL^2(B) \otimes_B \mathcal K$ and we will write $P_{\mathcal K} : \rL^2(M) \to \mathcal K$ for the orthogonal projection. Observe that $P_{\mathcal K} \in \langle M, M_1\rangle$ because $\mathcal K$ is invariant under the right action of $M_1$ on $\rL^2(M)$.

\begin{lem}\label{vanish}
Let $\Theta : \langle M, M_1\rangle \to Q \cap M_1$ be any conditional expectation such that $\Theta |_M$ is normal. Then $\Theta(uP_{\mathcal K}u^*) = 0$ for every $u \in \mathcal U(M_1)$. 
\end{lem}

\begin{proof}
Denote by $(M_1, \rL^2(M_1), J^{M_1}, \rL^2(M_1)_+)$ the standard form of $M_1$. Since $\langle M_1,B \rangle = (J^{M_1}BJ^{M_1})' \cap \mathbf B(\rL^2(M_1))$, the Hilbert space $\rL^2(M_1)$ is a $\langle M_1,B \rangle$-$B$-bimodule. Thus, the Hilbert space $\rL^2(M) \ominus \rL^2(M_1) = \rL^2(M_1) \otimes_B \mathcal K$ is naturally endowed with a structure of $\langle M_1, B\rangle$-$M_1$-bimodule. We denote by $\pi_0 : \langle M_1,B \rangle \to \mathbf B(\rL^2(M) \ominus \rL^2(M_1))$ the unital faithful normal $\ast$-representation arising from the left action of $\langle M_1,B \rangle$ on $\rL^2(M) \ominus \rL^2(M_1)$. Using the identification
$$\mathbf B(\rL^2(M) \ominus \rL^2(M_1)) \cong e_{M_1}^\perp \mathbf B(\rL^2(M))e_{M_1}^\perp$$
and precomposing with $\pi_0$, we obtain a nonunital normal $\ast$-representation $\pi : \langle M_1,B \rangle \to \mathbf B(\rL^2(M))$ such that:
\begin{itemize}
\item the range of $\pi$ is contained in $\langle M,M_1\rangle$,
\item $\pi(e_B) = P_{\mathcal K}$ and
\item $\pi(x) = x e_{M_1}^\perp$ for every $x \in M_1 \subset M$.
\end{itemize}
It follows that $\Psi := \Theta \circ \pi : \langle M_1,B\rangle \to Q \cap M_1$ is a $(Q \cap M_1)$-$(Q \cap M_1)$-bimodular completely positive map. We claim that $\Psi$ is normal on $M_1$. Indeed, let $\psi \in (Q \cap M_1)_\ast$ be any positive linear functional. Let $(x_i)_{i \in I}$ be any net in $M_1$ such that $x_i \to x$ $\sigma$-strongly as $i \to \infty$. By Cauchy--Schwarz inequality applied to $\psi \circ \Theta$, we have
\begin{align*}
|(\psi \circ \Psi) (x - x_i)| & = |(\psi \circ \Theta \circ \pi)(x - x_i)| \\
& = |(\psi \circ \Theta)(e_{M_1}^\perp(x - x_i))| \\
&\leq \|e_{M_1}^\perp\|_{\psi \circ \Theta} \, \|x - x_i\|_{\psi \circ \Theta}  \to 0 \quad \text{as } i \to \infty.
\end{align*}
This shows that $\psi \circ \Psi$ is $\sigma$-strongly continuous on $M_1$ and hence $\psi \circ \Psi$ is normal on $M_1$. Since this holds true for every positive linear functional $\psi \in (Q \cap M_1)_\ast$, we infer that $\Psi$ is normal on~$M_1$. 

Since $Q \cap M_1 \npreceq_{M_1} B$, Theorem \ref{thm-intertwining} $(\rm iv)$ implies that $\Psi |_{M_1e_BM_1} = 0$. In particular, for every $u \in \mathcal U(M_1)$, we obtain that 
$$\Theta(uP_{\mathcal K}u^*) = \Theta (u e_{M_1}^\perp \, P_{\mathcal K} \, (u e_{M_1}^\perp)^*)= \Theta(\pi (u) \, \pi(e_B) \, \pi(u^*)) =  \Theta(\pi (u e_B u^*)) = \Psi (ue_Bu^*) = 0.$$
This finishes the proof of the lemma.
\end{proof}

\begin{proof}[Proof of the main theorem]
Since $Q$ is amenable relative to $M_1$ inside $M$ and since $Q \subset M$ is with expectation, there exists a conditional expectation $\Phi : \langle M, M_1\rangle \to Q$ such that $\Phi |_M$ is faithful and normal. 

{\bf Claim.} We have $\Phi(x) = \Phi( \rE_{M_1}(x))$ for every $x \in M$. 

Indeed, fix a faithful normal conditional expectation $\rF : M \to Q \cap M_1$ and put $\Theta := \rF \circ \Phi : \langle M, M_1\rangle \to Q \cap M_1$. Observe that $\Theta$ is a conditional expectation such that $\Theta|_M$ is normal. By Lemma \ref{vanish}, since $Q \cap M_1 \npreceq_{M_1} B$, we have $\rF(\Phi(uP_{\mathcal K} u^*)) = \Theta(uP_{\mathcal K}u^*) = 0$ for every $u \in \cU(M_1)$. Since $\rF$ is faithful, we obtain $\Phi(uP_{\mathcal K}u^*)=0$ for every $u \in \mathcal U(M_1)$. In particular, the projection $1 - uP_{\mathcal K}u^*$ belongs to the multiplicative domain of $\Phi$ for every $u \in \mathcal U(M_1)$.

It suffices to prove the claim for $x$ being a word of the form $x = u \, x_1 \cdots x_p \, v \in M \ominus M_1$ where $p\geq 1$, $x_1 \cdots x_p$ is a reduced word starting and ending with letters $x_1,x_p \in M_2^\circ$ and $u,v \in \cU(M_1)$. Indeed, firstly observe that $\Phi(x) = \Phi( \rE_{M_1}(x))$ for every $x \in M_1$. Secondly, the linear span of such words as above is $\sigma$-strongly dense in $M \ominus M_1$ and $\Phi|_M$ and $\Phi \circ \rE_{M_1}$ are both normal on $M$.

Fix such a word $x = u \, x_1 \cdots x_p \, v \in M \ominus M_1$ as above. We will prove that $\Phi(x) = 0$. Observe that the subspace $\rL^2(M) \ominus \mathcal K$ is the closure of the linear span of the elements $y_1 \xi_\varphi$ with $y_1 \in M_1$ and $y_1 \cdots y_n \xi_\varphi$ where $n \geq 2$ and $y_1 \cdots y_n$ is a reduced word in $M$ starting with a letter $y_1 \in M_1^\circ$. Since $x_1 \cdots x_p$ is a reduced word starting and ending with letters $x_1,x_p \in M_2^\circ$, we obtain
\[ x_1 \cdots x_p \, (1-P_{\mathcal K}) = P_{\mathcal K} \, x_1 \cdots x_p \, (1-P_{\mathcal K}),\]
from which we infer the equality
\[(1 - u P_{\mathcal K} u^*) \, x \, (1 - v^* P_{\mathcal K} v) = u \, (1-P_{\mathcal K}) \, x_1 \cdots x_p \, (1-P_{\mathcal K}) \, v = 0.\]
Applying $\Phi$ to this equality and using the fact that $1 - u P_{\mathcal K} u^*$ and $1 - v^* P_{\mathcal K} v$ belong to the multiplicative domain of $\Phi$, we obtain
\[\Phi(x) = \Phi(1 - u P_{\mathcal K} u^*) \, \Phi(x) \,\Phi(1 - v^* P_{\mathcal K} v) = \Phi((1 - u P_{\mathcal K} u^*)\, x \, (1 - v^* P_{\mathcal K} v)) =  0.\]
This finishes the proof of the claim.

Now fix a faithful state $\psi \in M_\ast$ such that $\psi = \psi \circ (\Phi |_M)$. The above claim shows that $\psi = \psi \circ \rE_{M_1}$. For every $x \in Q$, we have
\[ \Vert x \Vert_\psi =  \Vert \Phi(x) \Vert_\psi = \Vert \Phi \circ \rE_{M_1}(x) \Vert_\psi \leq \Vert \rE_{M_1}(x) \Vert_\psi \leq  \Vert x \Vert_\psi.\]
This shows that $\Vert \rE_{M_1}(x) \Vert_\psi = \Vert x \Vert_\psi$ and hence $\Vert x - \rE_{M_1}(x)\Vert_\psi^2 = \Vert x \Vert_\psi^2  - \Vert \rE_{M_1}(x)\Vert_\psi^2 = 0$. We conclude that $x = \rE_{M_1}(x) \in M_1$ for every $x \in Q$, that is, $Q \subset M_1$.
\end{proof}

\begin{rem}
Let us mention that our main theorem yields an analogous result for HNN extensions of von Neumann algebras. To avoid technicalities, we only formulate it in the tracial setting. Following \cite[Section 2]{Ue07}, for any inclusion of tracial von Neumann algebras $N \subset M$ and any trace preserving embedding $\theta: N \hookrightarrow M$, denote by $\HNN( M, N, \theta)$ the corresponding HNN extension. Using our main theorem and \cite[Proposition 3.1]{Ue07}, we can show that for any  von Neumann subalgebra $Q \subset \HNN( M, N, \theta)$ which is amenable relative to $M$ inside $\HNN(M, N , \theta)$ and such that $Q \cap M \npreceq_{M} N$, we have  $Q \subset M$.
\end{rem}

\begin{rem}
Recall that when a probability measure preserving equivalence relation $\cR$ defined on a standard probability space $(X , \mu)$ splits as an amalgamated free product $\cR_1 \ast_{\cR_0} \cR_2$ in the sense of \cite[D\'efinition IV.6]{Ga99}, the associated von Neumann algebra satisfies $\rL(\cR) = \rL(\cR_1) \ast_{\rL(\cR_0)} \rL(\cR_2)$. Hence our main theorem shows any amenable subequivalence relation of $\cR$ that has a sufficiently large intersection with $\mathcal R_1$ must be contained in $\mathcal R_1$. In the case when the amalgam $\cR_0$ is the trivial relation, such a result follows from \cite[Th\'eor\`eme 1]{Al09}. However, our result is more general as it applies to arbitrary amalgams $\mathcal R_0$.
\end{rem}


\bibliographystyle{plain}

\end{document}